\documentclass[11pt]{amsart}
\usepackage[english]{babel}
\usepackage{amsmath}
\usepackage{amssymb}
\usepackage{amsthm}
\usepackage{libertine}
\usepackage[all]{xy}
 \newtheorem{thm}{Theorem}[section]
 \newtheorem{cor}[thm]{Corollary}
 \newtheorem{lem}[thm]{Lemma}
 \newtheorem{prop}[thm]{Proposition}
  
 \theoremstyle{definition}
 \newtheorem{defn}[thm]{Definition}
 \theoremstyle{remark}
 \newtheorem{rem}[thm]{Remark}
 
 \numberwithin{equation}{section}

\newcommand{\scal}[1]{\left<#1\right>}
\newcommand{\Hq}{\mathbb H}
\newcommand{\Sq}{\mathbb S}

\newcommand{\N}{\mathbb{N}}

\newcommand{\R}{\mathbb{R}}      

\newcommand{\C}{\mathbb{C}}

\title[On the quaternionic short-time Fourier and Segal-Bargmann transforms]{On the quaternionic short-time Fourier and Segal-Bargmann transforms}

\begin{document}
\date{}
\author{Antonino De Martino, Kamal Diki}
\maketitle
\begin{abstract}
In this paper, we study a special one dimensional quaternion short-time Fourier transform (QSTFT). Its construction is based on the slice hyperholomorphic Segal-Bargmann transform. We discuss some basic properties and prove different results on the QSTFT such as Moyal formula, reconstruction formula and Lieb's uncertainty principle. We provide also the reproducing kernel associated to the Gabor space considered in this setting.
 \end{abstract}

\noindent AMS Classification: 44A15, 30G35, 42C15

\noindent {\em Key words}: 1D quaternion Fourier transform, Segal-Bargmann transform, Short-time Fourier transform, Quaternions, Slice hyperholomorphic functions.

\section{Introduction}
Recently there has been an increased interest in the generalization of integral transforms to the quaternionic and Clifford settings. Such kind of transforms are widely studied, since they help in the analysis of vector-valued signals and images. In the survey \cite{CK} it is explained that some hypercomplex signals are useful tools for extracting intrinsically 1D-features from images. The reader can find other motivations for studying the extension of time frequency-analysis to quaternions in \cite{CK} and the references therein.
In the survey \cite{D} the author states that this research topic is based on three main approaches: the eigenfunction approach, the generalized roots of $-1$ approach and the spin group approach.

Using the second one a quaternionic short-time Fourier transform in dimension 2 is studied in \cite{BA}. In the paper \cite{D1} the same transform is defined in a Clifford setting for even dimension more than two.
In this paper we introduce an extension of the short-time Fourier transform in a quaternionic setting in dimension one.    

To this end,  we fix a property that relates the complex short-time Fourier transform and the complex Segal-Bargmann transform:
\begin{equation}
V_{ \varphi} f(x, \omega)=e^{-\pi i x \omega} Gf(\bar{z}) e^{ \frac{- \pi |z|^2}{2}},
\end{equation}
where $ V_{\varphi}$ is the complex short-time Fourier transform with respect to the Gaussian window $ \varphi$ (see \cite[Def. 3.1]{G})
and $Gf(z)$ denotes the complex version of the Segal-Bargmann transform according to \cite{G}. To achieve our aim we use the quaternionc analogue of the Segal-Bargmann transform studied in \cite{DG}. This integral transform is used also in \cite{DKS}  to study some quaternionic Hilbert spaces of Cauchy-Fueter regular functions. In \cite{CSS} and \cite{PSS} the authors introduce some special modules of monogenic functions of Bargmann-type in Clifford analysis.

In order to present our results, we adopt the following structure: in section 2 we collect some basic definitions and preliminaries.
In section 3, we prove some new properties of the quaternionic Segal-Bargmann transform. In particular we deal with an unitary property and give a characterization of the range of the Schwartz space. Moreover, we provide some calculations related to the position and the momentum operators.

In section 4, we give a brief overview of the 1D Fourier transform \cite{ELS} and show a Plancherel theorem in this framework.

In section 5, we define the 1D QSTFT in the following way
\begin{eqnarray*}
\mathcal{V}_{\varphi}f(x, \omega)= e^{-I \pi x \omega} \mathcal{B}_\Hq^S(f) \biggl(\frac{\bar{q}}{\sqrt{2}} \biggl) e^{-\frac{|q|^2 \pi}{2}}\\
\end{eqnarray*}

where $ \mathcal{B}_\Hq^S$ is the quaternionic Segal-Bargmann transform.

Using some properties of $ \mathcal{B}_\Hq^S$ we prove an isometric relation for the 1D QSTFT and a Moyal formula. These implies the following reconstruction formula
$$f(y)=\displaystyle 2^{-\frac{1}{4}} \int_{\R^2} e^{2\pi I \omega y}\mathcal{V}_{\varphi}f(x,\omega)e^{-\pi(y-x)^2}dx d\omega, \textbf{  } \forall y\in\R.$$
From this follows that the adjoint operator
defines a left inverse. Furthermore, it gives the possibility to write the 1D QSTFT using the reproducing kernel associated to the Gabor space
$$\mathcal{G}^{\varphi}_{\Hq}:=\lbrace{\mathcal{V}_\varphi f, \textbf{  } f\in L^2(\R,\Hq)}\rbrace.$$
Finally, we show that the 1D QSTFT follows a Lieb's uncertainty principle, some classical uncertainty principles for quaternionic linear operators in quaternionic Hilbert spaces were considered in \cite{XR}.
\section{Preliminaries}
In 2006 a new approach to quaternionic regular functions was introduced and then extensively studied in several directions, and it is nowadays widely developed \cite{ACS1,CSS1, CSS2, GSS}.
This new theory contains polynomials and power series with quaternionic coefficients in the right, contrary to the Fueter theory of regular functions defined by means of the Cauchy-Riemann Fueter differential operator. The meeting point between the two function theories comes from an idea of Fueter in the thirties and next developed later by Sce \cite{S} and by Qian \cite{Q}. This connection holds in any odd dimension (and in quaternionic case) and has been explained in \cite{CSS3} in the language of slice regular functions with values in the quaternions and slice monogenic functions with values in a Clifford algebra. The inverse map has been studied in \cite{CSS4} and still holds in any odd dimension. Moreover, the theory of slice regular functions have several applications in operator theory and in Mathematical Physics. The spectral theory of the S-spectrum is a natural tool for the formulation of quaternionic quantum mechanics and for the study of new classes of fractional diffusion problems, see \cite{CGK,CG},
and the references therein.
To make the paper self-contained,  we briefly revise here the basics of the slice regular functions. Let $ \mathbb{H}$ denote the quaternion algebra with its standard basis $\{1,i,j,k\}$ satisfying the  multiplication $i^2=j^2=k^2=ijk=-1$, $ij=-ji=k$, $jk=-kj=i$ and $ki=-ik=j$. For $ q \in \mathbb{H}$, we write $q=x_0+x_1i+x_2j+x_3k$ with $x_0,x_1,x_2, x_3 \in \mathbb{R}$. With respect to the quaternionic conjugate defined to be $ \bar{q}=x_0-x_1i-x_2j-x_3k= \hbox{Re}(q)- \hbox{Im}(q)$, we have $ \overline{pq}= \bar{q} \bar{p}$ for $p,q \in \mathbb{H}$. The modulus of $q$ is defined to be $|q|= \sqrt{q \bar{q}} =\sqrt{x_0^2+x_1^2+x_2^2+x_3^2}$. In particular we have $| \hbox{Im}q|= \sqrt{x_1^2+x_2^2+x_3^2}$. Let $ \mathbb{S}= \{q \in \mathbb{H}; q^2=-1\}$ be the unit sphere of imaginary units in $\mathbb{H}$. Note that any $ q \in \Hq \setminus \mathbb{R}$ can be written in a unique way as $q=x+Iy$ for some real numbers $x$ and $y>0$, and imaginary unit $I \in \mathbb{S}$. For every given $I \in \mathbb{S}$ we define $ \mathbb{C}_I= \mathbb{R}+ \mathbb{R}I.$ It is isomorphic to the complex plane $ \mathbb{C}$ so that it can be considered as a complex plane in $\Hq$ passing through 0,1 and I. Their union is the whole space of quaternions

$$ \mathbb{H}= \bigcup_{I \in \mathbb{S}} \mathbb{C}_I= \bigcup_{I \in \mathbb{S}} \mathbb{R}+ \mathbb{R}I$$

\begin{defn}
A real differentiable function $f: \Omega \to \mathbb{H}$, on a given domain $ \Omega \subset \mathbb{H}$, is said to be a (left) slice regular function if, for every $ I \in \mathbb{S}$, the restriction $f_I$ to $ \mathbb{C}_I$, with variable $q=x+Iy$, is holomorphic on $ \Omega_I:= \Omega \cap \mathbb{C}_I$, that is, it has continuous partial derivatives with respect to $x$ and $y$ and the function $ \overline{\partial_I}f: \Omega_I \to \mathbb{H}$ defined by
$$ \overline{\partial_I}f(x+Iy):= \frac{1}{2} \biggl( \frac{\partial }{\partial x}+I \frac{\partial}{\partial y} \biggl) f_I(x+yI)$$
vanishes identically on $ \Omega_I$. The set of slice regular functions will be denoted by $ \mathcal{SR}(\Omega)$.
\end{defn}
Characterization of slice regular functions on a ball $B=B(0,R)$ centred at the origin is given in \cite{GSS}. Namely we have
\begin{lem}
A given $ \mathbb{H}$- valued function$f$ is slice regular on $B(0,R) \subset \mathbb{H}$ if and only if it has a series expansion of the form
$$ f(q)= \sum_{n=0}^\infty \frac{q^n}{n!} \frac{\partial^n f}{\partial x^n}(0),$$
converging on $B(0,R)= \{q \in \mathbb{H}; |q|<R\}$.
\end{lem}

\begin{defn}
Let $f: \Omega \to \mathbb{H}$ be a regular function. For each $I \in \mathbb{S}$, the $I$-derivative of $f$ is defined as
$$ \partial_I f(x+Iy):= \frac{1}{2} \biggl( \frac{\partial }{\partial x}-I \frac{\partial}{\partial y} \biggl) f_I(x+yI),$$
on $ \Omega_I$. The slice derivative of $f$ is the function $ \partial_Sf: \Omega \to \mathbb{H}$ defined by $ \partial_If$ on $ \Omega_I$, for all $I \in \mathbb{S}.$
\end{defn}

In all the paper we will make use of the Hilbert space $ L^2(\mathbb{R}, dx)=L^{2}(\mathbb{R}, \mathbb{H})$, consisting of all the square integrable $\Hq$-valued functions with respect to
$$\scal{\psi,\phi}=\displaystyle \int_\R \overline{\phi(t)}\psi(t)dt.$$
In \cite{ACS} the authors introduce the slice hyperholomorphic quaternionic Fock space $\mathcal{F}^{2,\nu}_{Slice}(\Hq)$,  defined for a given $I\in{\mathbb{S}}$  to be
$$\mathcal{F}^{2,\nu}_{Slice}(\Hq):=\lbrace{f\in{\mathcal{SR}(\Hq); \, \displaystyle  \int_{\C_I}\vert{f_I(p)}\vert^2 e^{-\nu\vert{p}\vert^2}d\lambda_I(p) <\infty}}\rbrace,$$
where $\nu>0, $ $f_I = f|_{\C_I}$ and $d\lambda_I(p)=dxdy$ for $p=x+yI$.
The right $\Hq$-vector space $\mathcal{F}^{2,\nu}_{Slice}(\Hq)$ is endowed with the inner product
\begin{equation}\label{spfg}
\scal{f,g}_{\mathcal{F}^{2,\nu}_{Slice}(\Hq)} = \int_{\C_I}\overline{g_I(q)}f_I(q)e^{-\nu\vert{q}\vert^2} d\lambda_I(q), \textbf{  } \forall f,g\in{\mathcal{F}^{2,\nu}_{Slice}(\Hq)}.
\end{equation}
The associated norm is given by
$$\Vert{f}\Vert_{\mathcal{F}^{2,\nu}_{Slice}(\Hq)}^2= \int_{\C_I}\vert{f_I(q)}\vert^2 e^{-\nu\vert{q}\vert^2}d\lambda_I(q).$$
This quaternionic Hilbert space does not depend on the choice of the imaginary unit $I$. An associated Segal-Bargmann transform was studied  in \cite{DG} by considering the kernel function obtained by means of generating function related to the normalized weighted Hermite functions $$\mathcal{A}_\Hq^S(q,x):=\displaystyle \sum_{k=0}^\infty f_k^\nu(q)\psi_k^\nu(x)=\left(\frac{\nu}{\pi}\right)^{\frac{3}{4}}e^{-\frac{\nu}{2}(q^2+x^2)+\nu\sqrt{2}qx}, \textbf{ } \forall (q,x)\in \Hq\times \R$$
where $\psi_k^\nu$ denote the normalized weighted Hermite functions:
$$\psi_k^\nu(x):= \frac{h_k^{\nu}(x)}{\| h_k^{\nu}(x) \|_{L^{2}(\mathbb{R}, \mathbb{H})}} = \frac{(-1)^{k} e^{\frac{\nu}{2} x^2} \frac{d^k}{dx^k} \bigl( e^{-\nu x^2} \bigl)}{2^{k/2} \nu^{k/2} (k!)^{1/2} \pi^{1/4} \nu^{-1/4}},$$
and
$$\displaystyle f_k^\nu(q):= \frac{e_k(q)}{\| e_k(q) \|_{\mathcal{F}^{2,\nu}_{Slice}(\Hq)}}=\frac{q^k}{||q^k||_{\mathcal{F}^{2,\nu}_{Slice}(\Hq)}}=\sqrt{\frac{\nu^{k+1}}{\pi k!}}q^k, \textbf{  } \forall k \geq 0, $$ are the normalized quaternionic monomials which constitute an orthonormal  basis of $\mathcal{F}^{2,\nu}_{Slice}(\Hq)$.
Then, for any quaternionic valued function $\varphi$ in $L^2(\R,\Hq)$ the slice hyperholomorphic Segal-Bargmann transform is defined by
\begin{equation}
\label{Barg}
\displaystyle \mathcal{B}_\Hq^S(\varphi)(q)=\int_\R \mathcal{A}_\Hq^S(q,x) \varphi(x)dx.
\end{equation}
In particular, most of our calculations later will be with a fixed parameter even $\nu=1$ or $\nu=2\pi$.

\section{Further properties of the quaternionic Segal-Bargmann transform}
In this section we prove some new properties of the quaternionic Segal-Bargmann transform. We start from an unitary property which is not found in literature in the following explicit form.
\begin{prop}
Let $f,g \in L^{2}(\mathbb{R}, \mathbb{H})$. Then, we have
\begin{equation}
\label{NL}
\langle \mathcal{B}_\Hq^S(f), \mathcal{B}_\Hq^S(g) \rangle_{\mathcal{F}^{2,\nu}_{Slice}(\mathbb{H})}= \langle f,g \rangle_{L^{2}(\mathbb{R}, \mathbb{H})}.
\end{equation}
\end{prop}
\begin{proof}
Any $f,g \in  L^{2}(\mathbb{R}, \mathbb{H})$ can be expanded as
$$ f(x)= \sum_{k \geq 0} h_{k}^{\nu}(x) \alpha _k,$$
$$ g(x)= \sum_{k \geq 0} h_{k}^{\nu}(x) \beta_k,$$
where $(\alpha _k)_{k \in \mathbb{N}}, (\beta _k)_{k \in \mathbb{N}} \subset \mathbb{H}$.
\begin{eqnarray}
\label{W3}
\nonumber
\langle f,g \rangle_{L^{2}(\mathbb{R}, \mathbb{H})} &=& \int_{\mathbb{R}} \overline{g(x)} f(x) \, dx= \sum_{k \geq 0} \int_{\mathbb{R}} \overline{h_{k}^{\nu}(x) \beta_k} h_{k}^{\nu}(x) \alpha _k \, dx \\
&=& \sum_{k \geq 0} \overline{\beta_k} \biggl( \int_{\mathbb{R}} \overline{h_{k}^{\nu}(x)} h_{k}^{\nu}(x) \, dx \biggl) \alpha_k \\
&=& \sum_{k \geq 0} \| h_{k}^{\nu}(x) \|_{L^{2}(\mathbb{R}, \mathbb{H})}^2  \overline{\beta_k} \alpha_k. \nonumber
\end{eqnarray}
On the other way, since
$$ \langle f, h_k^\nu \rangle_{L^{2}(\mathbb{R}, \mathbb{H})}= \sum_{j \geq 0} \biggl( \int_{\mathbb{R}} \overline{h_{k}^{\nu}(x)} h_{j}^{\nu}(x) \, dx \biggl) \alpha_j = \| h_{k}^{\nu}(x) \|_{L^{2}(\mathbb{R}, \mathbb{H})}^2 \alpha_k.$$
We have by \cite{DG}
\begin{eqnarray}
\label{W0}
\mathcal{B}_\Hq^S(f)(q) &=& \sum_{k \geq 0} e_k(q) \frac{\langle f, h_{k}^{\nu} \rangle_{L^{2}(\mathbb{R}, \mathbb{H})}}{ \| h_k^\nu(x) \|_{L^{2}(\mathbb{R}, \mathbb{H})} \| e_k(q) \|_{\mathcal{F}^{2, \nu}_{Slice}}} \\ \nonumber
&=& \sum_{k \geq 0} e_k(q) \frac{\| h_k^\nu(x) \|_2^2 }{ \| h_k^\nu(x) \|_{L^{2}(\mathbb{R}, \mathbb{H})} \| e_k(q) \|_{\mathcal{F}^{2, \nu}_{Slice}}}\alpha_k\\ \nonumber
&=& \sum_{k \geq 0} e_k(q) \frac{\| h_k^\nu(x) \|_{L^{2}(\mathbb{R}, \mathbb{H})} }{ \| e_k(q) \|_{\mathcal{F}^{2, \nu}_{Slice}}} \alpha_k.	
\end{eqnarray}
Using the same calculus we obtain
\begin{equation}
\label{W1}
\overline{\mathcal{B}_\Hq^S(g)(q)}= \sum_{k \geq 0} \frac{\| h_k^\nu(x) \|_{L^2 (\mathbb{R}, \mathbb{H})}}{\| e_k(q) \|_{\mathcal{F}^{2, \nu}_{Slice}}} \overline{e_k(q) \beta_k}.
\end{equation}
By putting together \eqref{W0} and \eqref{W1} we obtain
\begin{eqnarray}
\nonumber
\label{W2}
\langle \mathcal{B}_\Hq^S(f), \mathcal{B}_\Hq^S(g) \rangle_{\mathcal{F}^{2,\nu}_{Slice}(\mathbb{H})} &=& \sum_{k \geq 0} \int_{\mathbb{C}_I} \| h_k^\nu(x) \|_{L^{2}(\mathbb{R}, \mathbb{H})}^2 \overline{\beta_k} \frac{\overline{e_k(q)}}{\| e_k(q) \|_{\mathcal{F}^{2, \nu}_{Slice}}} \cdot \\ \nonumber
&& \cdot \frac{e_k(q)}{\| e_k(q) \|_{\mathcal{F}^{2, \nu}_{Slice}}} \alpha_k e^{- \nu |q|^2} \, d \lambda_I(q)\\ \nonumber
&=& \sum_{k \geq 0} \| h_k^\nu(x) \|_{L^{2}(\mathbb{R}, \mathbb{H})}^2 \overline{\beta_k} \biggl( \int_{\mathbb{C}_I} \frac{\overline{e_k(q)}}{\| e_k(q) \|_{\mathcal{F}^{2,  \nu}_{Slice}}} \cdot \\ \nonumber
&& \cdot \frac{e_k(q)}{\| e_k(q) \|_{\mathcal{F}^{2, \nu}_{Slice}}} e^{- \nu |q|^2}  \, d \lambda_I(q) \biggl) \alpha_k\\ \nonumber
&=& \sum_{k \geq 0} \| h_k^\nu(x) \|_{L^{2}(\mathbb{R}, \mathbb{H})}^2 \overline{\beta_k}  \frac{1}{\| e_k(q) \|_{\mathcal{F}^{2, \nu}_{Slice}}^2} \cdot \\ \nonumber
&& \cdot \biggl( \int_{\mathbb{C}_I} \overline{e_k(q)}e_k(q) e^{- \nu |q|^2}  \, d \lambda_I(q) \biggl) \alpha_k\\  \nonumber
&=& \sum_{k \geq 0} \| h_k^\nu(x) \|_{L^{2}(\mathbb{R}, \mathbb{H})}^2 \overline{\beta_k}  \frac{1}{\| e_k(q) \|_{\mathcal{F}^{2,  \nu}_{Slice}}^2} \| e_k(q) \|_{\mathcal{F}^{2,  \nu}_{Slice}}^2  \alpha_k\\
&=& \sum_{k \geq 0} \| h_k^\nu(x) \|_{L^{2}(\mathbb{R}, \mathbb{H})}^2 \overline{\beta_k} \alpha_k
\end{eqnarray}
Finally, since \eqref{W3} and \eqref{W2} are equal we obtain the thesis.
\end{proof}
\begin{rem}
If $f=g$ in \eqref{NL} we have that the quaternionic Segal-Bargmann transform realizes an isometry from $L^2(\mathbb{R}, \mathbb{H})$ onto the slice hyperholomorphic Bargmann-Fock space $ \mathcal{F}^{2, \nu}_{Slice}(\mathbb{H})$, as proved in a different way in \cite[Thm. 4.6]{DG}
\end{rem}

\subsection{Range of the Schwartz space and some operators}
We characterize the range of the Schwartz space under the Segal-Bargmann transform with parameter $\nu=1$ in the slice hyperholomorphic setting of quaternions. We consider also some equivalence relations related to the position and momentum operators in this setting.
The quaternionic Schwartz space on the real line that we are considering in this framework is defined by
$$\mathcal{S}_\Hq(\R):=\lbrace{\psi:\R\longrightarrow\Hq \textbf{ }: \text{  } \sup_{x\in\R}\left|x^\alpha\frac{d^\beta}{dx^\beta}(\psi)(x)\right|<\infty},\textbf{ } \forall \alpha,\beta\in \mathbb{N} \rbrace.$$
For $I\in\Sq$, the classical Schwartz space is given by
$$\mathcal{S}_{\C_I}(\R):=\lbrace{\varphi:\R\longrightarrow\C_I;\textbf{ } : \text{  } \sup_{x\in\R}\left|x^\alpha\frac{d^\beta}{dx^\beta}(\varphi)(x)\right|<\infty, \textbf{ } \forall \alpha,\beta\in \mathbb{N} }\rbrace .$$
Clearly, we have that $$\mathcal{S}_{\C_I}(\R)\subset\mathcal{S}_\Hq(\R)\subset L^2_\Hq(\R).$$ Moreover, we prove the following
\begin{lem} \label{L1} Let $\psi:x\longmapsto \psi(x)$ be a quaternionic valued function. Let $I,J\in\mathbb{
S}$ be such that $I \perp J$. Then, $\psi\in\mathcal{S}_\Hq(\R)$ if and only if there exist $\varphi_1,\varphi_2\in \mathcal{S}_{\C_I}(\R)$ such that we have $$\psi(x)=\varphi_1(x)+\varphi_2(x)J, \textbf{ } \forall x\in\R.$$

\end{lem}

\begin{proof}
Let $\psi\in \mathcal{S}_\Hq(\R) $. Then, we can write $$\psi(x)=\varphi_1(x)+\varphi_2(x)J,$$ where $\varphi_1$ and $\varphi_2$ are $\C_I-$valued functions. Note that for all $\alpha,\beta\in \N$ we have $$\left|x^\alpha\frac{d^\beta}{dx^\beta}(\psi)(x)\right|^2=\left|x^\alpha\frac{d^\beta}{dx^\beta}(\varphi_1)(x)\right|^2+\left|x^\alpha\frac{d^\beta}{dx^\beta}(\varphi_2)(x)\right|^2.$$ In particular, this implies that $\psi\in \mathcal{S}_\Hq(\R) $ if and only if $\varphi_1,\varphi_2\in \mathcal{S}_{\C_I}(\R)$.
\end{proof}

Let us now denote by $\mathcal{SF}(\mathbb{H})$ the range of $\mathcal{S}_\Hq(\R)$ under the quaternionic Segal-Bargmann transform $\mathcal{B}_{\Hq}^S.$ Therefore, we have the following characterization of $\mathcal{SF}(\mathbb{H})$:
\begin{thm}
A function $f(q)=\displaystyle\sum_{k=0}^\infty q^kc_k$ belongs to $\mathcal{SF}(\mathbb{H})$ if and only if $$\sup_{k\in\mathbb{N}}|c_k|k^p\sqrt{k!}<\infty, \forall p >0.$$
i.e,
$$\mathcal{SF}(\mathbb{H})=\lbrace{\displaystyle\sum_{k=0}^\infty q^k c_k, \textbf{ } c_k\in\Hq \text{ and  }\sup_{k\in\mathbb{N}}|c_k|k^p\sqrt{k!}<\infty, \forall p >0}\rbrace.$$
\end{thm}
\begin{proof}
Let $f\in\mathcal{SF}(\mathbb{H}),$ then by definition $f=\mathcal{B}_\Hq^S \psi$ where $\psi\in\mathcal{S}_\Hq(\R)$. Let $I,J\in \Sq$,  be such that $I \perp J$. Thus, Lemma \ref{L1} implies that  $$\psi(x)=\varphi_1(x)+\varphi_2(x)J,$$
where $\varphi_1,\varphi_2\in \mathcal{S}_{\C_I}(\R)$. Therefore, we have $$\mathcal{B}_\Hq^S(\psi)(q)=\mathcal{B}^S_\Hq(\varphi_1)(q)+\mathcal{B}^S_\Hq(\varphi_2)(q)J.$$
Then, we take the restriction to the complex plane $\C_I$ and get:
$$\mathcal{B}_\Hq^S(\psi)(z)=\mathcal{B}_{\C_I}(\varphi_1)(z)+\mathcal{B}_{\C_I}(\varphi_2)(z)J, \textbf{  }\forall z\in\C_I,$$
where the complex Bargmann transform (see \cite{B}) is given by

$$\displaystyle\mathcal{B}_{\C_I}(\varphi_l)(z)=\frac{1}{\pi^{\frac{3}{4}}}\int_\R e^{-\frac{1}{2}(z^2+x^2)+\sqrt{2}zx} \varphi_l(x)dx, \textbf{  } l=1,2.$$
In particular, we set $f_I:=\mathcal{B}_\Hq^S(\psi)$, $f_1:=\mathcal{B}_{\C_I}(\varphi_1)$ and $f_2:=\mathcal{B}_{\C_I}(\varphi_2)$.
 Then, we have $f_1,f_2\in\mathcal{SF}(\mathbb{C}_I) $. Thus, by applying the classical result in complex analysis, see \cite{N} we have $$f_1(z)=\displaystyle\sum_{n=0}^{\infty}a_nz^n \text{ and } f_2(z)=\sum_{n=0}^\infty b_nz^n, \text{  } \forall z\in\C_I.$$ Moreover, for all $p>0$ the following conditions hold $$\sup_{n\in\N}|a_n|n^p\sqrt{n!}<\infty \text{ and }\sup_{n\in\N}|b_n|n^p\sqrt{n!}<\infty. $$ In particular, we have then $$f_I(z)=\displaystyle \sum_{n=0}^\infty a_nz^n+(\sum_{n=0}^\infty a_nz^n)J, \text{  }\forall z\in\C_I.$$

Therefore, $$f_I(z)=\displaystyle \sum_{n=0}^{\infty}z^nc_n \text{ with } c_n=a_n+b_nJ, \text{ for all } z\in\C_I.$$

Thus, by taking the slice hyperholomorphic extension we get $$f(q)=\displaystyle \sum_{n=0}^{\infty}q^nc_n, \textbf{ }\forall q\in\Hq.$$ Moreover, note that $c_n=a_n+b_nJ, n\in \N$. Then, $|c_n|\leq |a_n|+|b_n|$, $\forall n\in \N$. Thus, for all $p>0$, we have $$\sup_{n\in\N}|c_n|n^p\sqrt{n!}\leq\sup_{n\in\N}|a_n|n^p\sqrt{n!}+\sup_{n\in\N}|b_n|n^p\sqrt{n!}<\infty.$$
Finally, we conclude that

$$\mathcal{SF}(\mathbb{H})=\lbrace{f(q)=\displaystyle\sum_{k=0}^\infty q^k c_k, \textbf{ } c_k\in\Hq \text{ and  }\sup_{k\in\mathbb{N}}|c_k|k^p\sqrt{k!}<\infty, \forall p>0}\rbrace.$$
\end{proof}

Now, let us consider  on $L^2(\R,\Hq)=L^2_\Hq(\R)$ the position and momentum operators defined by $$X:\varphi\mapsto X\varphi(x)=x\varphi(x) \text{ and } D:\varphi\mapsto D\varphi(x)=\frac{d}{dx}\varphi(x).$$
Their domains are given respectively by

$$\mathcal{D}(X):=\lbrace{\varphi \in L^2_\Hq (\R); \text{ } X\varphi\in L^2_\Hq (\R)}\rbrace \text{ and } \mathcal{D}(D):=\lbrace{\varphi \in L^2_\Hq (\R); \text{ } D\varphi\in L^2_\Hq (\R)}\rbrace.$$
First, let us prove the following
\begin{lem}  \label{lem4.1}
For all $(q,x)\in \Hq\times\R,$ we have
$$\partial_S\mathcal{A}_\Hq^S(q,x)=(-q+\sqrt{2}x)\mathcal{A}_\Hq^S(q,x).$$
\end{lem}
\begin{proof}
Let $(q,x)\in\Hq\times\R$. Then, by definition of the quaternionic Segal-Bargmann kernel we can write $$\mathcal{A}_\Hq^S(q,x):=\pi^{-\frac{3}{4}}e^{-\frac{x^2}{2}}e^{-\frac{q^2}{2}}e^{\sqrt{2}qx}.$$
In this case, we can apply the Leibnitz rule with respect to the slice derivative and get $$\partial_S\mathcal{A}_\Hq^S(q,x)=\pi^{-\frac{3}{4}}e^{-\frac{x^2}{2}}\left(e^{-\frac{q^2}{2}}\partial_S(e^{\sqrt{2}xq})+\partial_S(e^{-\frac{q^2}{2}})e^{\sqrt{2}xq}\right).$$
However, using the series expansion of the exponential function and applying the slice derivative we know that $$\partial_S(e^{-\frac{q^2}{2}})=-qe^{-\frac{q^2}{2}} \text{ and }\partial_S(e^{\sqrt{2}xq})=\sqrt{2}xe^{\sqrt{2}xq}.$$
Therefore, we obtain $$\partial_S\mathcal{A}_\Hq^S(q,x)=(-q+\sqrt{2}x)A_\Hq^S(q,x).$$
\end{proof}
\begin{thm}
Let $\varphi\in \mathcal{D}(X)$. Then, we have
$$\left(\partial_S+q \right)\mathcal{B}_\Hq^S (\varphi)(q)=\sqrt{2}\mathcal{B}_\Hq^S(x\varphi)(q), \text{  } \forall q\in\Hq.$$
\end{thm}
\begin{proof}
Let $\varphi\in\mathcal{D}(X)$ and $q\in\Hq$. Then, we have $$\partial_S \mathcal{B}_\Hq^S(\varphi)(q)=\displaystyle \int_\R\partial_S\mathcal{A}_\Hq^S(q,x)\varphi(x)dx.$$
Therefore, using Lemma \ref{lem4.1} we obtain $$\partial_S \mathcal{B}_\Hq^S(\varphi)(q)=\displaystyle \sqrt{2}\mathcal{B}_\Hq^S(x\varphi)(q)-q\mathcal{B}_\Hq^S(\varphi)(q).$$
Finally, we get

$$\left(\partial_S+q \right)\mathcal{B}_\Hq^S (\varphi)(q)=\sqrt{2}\mathcal{B}_\Hq^S(x\varphi)(q), \text{  } \forall q\in\Hq.$$
\end{proof}
As a quick consequence, we have
\begin{cor} The position operator $X$ on $L^2_\Hq(\R)$ is equivalent to the operator $\displaystyle \frac{1}{\sqrt{2}}(\partial_S+q)$ on the space $\mathcal{F}^{2,1}_{Slice}(\Hq)$ via the quaternionic Segal-Bargmann transform $\mathcal{B}_\Hq^S$. In other words, for all $\varphi \in\mathcal{D}(X)$ we have 
$$X(\varphi)=(\mathcal{B}_{\Hq}^{S})^{-1} \frac{(\partial_S+q)}{\sqrt{2}}\mathcal{B}_{\Hq}^S(\varphi).$$
\end{cor}
On the other hand, we have also the following
\begin{thm}
We denote by $M_q: \varphi \longmapsto M_q\varphi(q)=q\varphi(q)$ the creation operator on $\mathcal{F}^{2,1}_{Slice}(\Hq)$. Then, we have $$(\mathcal{B}_{\Hq}^{S})^{-1}M_q\mathcal{B}_{\Hq}^S=\frac{1}{\sqrt{2}}(X-D)\text{ on } \mathcal{D}(X)\cap\mathcal{D}(D).$$
\end{thm}
\begin{proof}
Let $\varphi\in\mathcal{D}(X)\cap\mathcal{D}(D) $. Then, we have
 \[ \begin{split}
 \displaystyle  \mathcal{B}^S_\Hq(D\varphi)(q) & = \int_\R \mathcal{A}_\Hq^S(q,x) \frac{d}{dx}\varphi(x) dx \\
&=-\int_\R\frac{d}{dx}\mathcal{A}_\Hq^S(q,x)\varphi(x)dx.\\
\end{split}
\]
However, note that for all $(q,x)\in\Hq\times\R$, we have $$\displaystyle \frac{d}{dx}\mathcal{A}_\Hq^S(q,x)=(-x+\sqrt{2}q)\mathcal{A}_\Hq^S(q,x).$$
Therefore, $$\mathcal{B}^S_\Hq(D\varphi)(q)=\mathcal{B}^S_\Hq(x\varphi)(q)-\sqrt{2}q\mathcal{B}^S_\Hq(\varphi)(q).$$

Thus, we obtain $$M_q\mathcal{B}^S_\Hq(\varphi)=\mathcal{B}_\Hq^S\left(\frac{1}{\sqrt{2}}(X-D)\right)(\varphi).$$
Finally, we just need to apply $(\mathcal{B}_{\Hq}^{S})^{-1}$ to complete the proof.
\end{proof}
\section{1D quaternion Fourier transform}
In this section, we study the one dimensional quaternion Fourier transforms (QFT). Namely, we are considering here the 1D left sided QFT studied in chapter 3 of the book \cite{ELS}. In order to have less problems with computations we add $ - 2 \pi$ to the exponential.
\begin{defn}
The left sided 1D quaternionic Fourier transform of a quaternion valued signal $\psi:\R\longrightarrow \Hq$ is defined on $L^1(\R;dx)=L^1(\R;\Hq)$ by
  $$\mathcal{F}_I(\psi)(\omega) = \int_{\R} e^{-2 \pi I \omega t} \psi(t) dt $$
  for a given $I\in \mathbb{S}$.
  Its inverse is defined by $$\overset{\sim}{\mathcal{F}_I}(\phi)(t) = \int_{\R} e^{2 \pi I \omega t} \phi(\omega) d\omega. $$
\end{defn}
Let $J\in\mathbb{S}$ be such that $J\perp I$. We can split the signal $\psi$ via symplectic decomposition into simplex and perplex parts with respect to $I$ such that we have:$$\psi(t)=\psi_1(t)+\psi_2(t)J $$ where $\psi_1(t),\psi_2(t)\in \C_I.$
The left sided 1D QFT of $\psi$ becomes $$\mathcal{F}_I(\psi)(\omega) = \int_{\R} e^{-2 \pi I \omega t} \psi_1(t) dt+\int_{\R} e^{-2 \pi I \omega t} \psi_2(t) dt J$$  so that $$\mathcal{F}_I(\psi)(\omega)=\mathcal{F}_I(\psi_1)(\omega)+\mathcal{F}_I(\psi_2)(\omega)J.$$
 According to \cite{ELS}, most of the properties may be inherited from the classical complex case thanks to the equivalence between $\C_I$ and the standard complex plane and the fact that QFT can be decomposed into a sum of complex subfield functions.

Now, we define two fundamental operators for time-frequency analysis.
\newline
\newline
\emph{Translation}
$$ \tau_x \psi (t):=\psi(t-x) \qquad x \in \mathbb{R}.$$
\newline
\newline
\emph{Modulation}
$$ M_{\omega} \psi(t)= e^{2 \pi I \omega t} \psi(t), \qquad \omega \in \mathbb{R}.$$
As in the classical case we have a commutative relation between the two operators.
\begin{lem}
Let $ \psi$ be a function in $L^2(\mathbb{R}, \mathbb{H})$ then we have
\begin{equation}
\label{Com}
\tau_x M_\omega \psi (t)=e^{-2 \pi I \omega x} M_\omega \tau_x \psi(t), \qquad \omega, x \in \mathbb{R}.
\end{equation}
\end{lem}
\begin{proof}
It is just a matter of computations
\begin{eqnarray*}
\tau_x M_\omega \psi(t) &=& M_\omega \psi(t-x)=e^{2 \pi I \omega (t-x)}\psi(t-x)\\
&=&  e^{2 \pi I \omega t} e^{-2 \pi I \omega x} \psi(t-x)\\
&=& e^{-2 \pi I \omega x} e^{2 \pi I \omega t} \psi(t-x)\\
&=& e^{-2 \pi I \omega x} M_\omega \tau_x \psi(t).
\end{eqnarray*}
\end{proof}
From \cite[Table 3.2]{ELS} we have the following properties
\begin{equation}
\label{F1}
\mathcal{F}_I (\tau_x \psi)=M_{-x} \mathcal{F}_I( \psi),
\end{equation}

\begin{equation}
\label{F2}
\mathcal{F}_I (M_\omega \psi)= \tau_\omega \mathcal{F}_I (\psi).
\end{equation}
From \eqref{F1} and \eqref{F2} follow easily that
\begin{equation}
\label{F3}
\mathcal{F}_I(M_\omega \tau_x \psi)=\tau_ \omega M_{- x} \mathcal{F}_I (\psi).
\end{equation}

Then, we prove a version of the Plancherel theorem for 1D QFT.
\begin{thm}
Let $\phi,\psi\in L^2(\R,\Hq)$. Then, we have $$\scal{\mathcal{F}_I(\phi),\mathcal{F}_I(\psi)}_{L^2(\R,\Hq)}=\scal{\phi,\psi}_{L^2(\R,\Hq)}.$$ In particular,  for any $\phi\in L^2(\R,\Hq)$ we have $$||\mathcal{F}_I(\phi)||_{L^2(\R,\Hq)}=||\phi||_{L^2(\R,\Hq)}.$$
\end{thm}
\begin{proof}
Let $\phi,\psi\in L^2(\R,\Hq)$. By inversion formula for the 1D QFT, see \cite{ELS}, we have
 $$\phi(\omega)=\overset{\sim}{\mathcal{F}_I}(\mathcal{F}_I(\phi))(\omega), \textit{  }\forall \omega\in\R.$$
 Thus, direct computations using Fubini's theorem lead to
\[ \begin{split}
 \displaystyle  \scal{\phi,\psi}_{L^2(\R,\Hq)} & = \int_\R \overline{\psi(\omega)}\left(\int_\R e^{2\pi I\omega t}\mathcal{F}_I(\phi)(t)dt \right)d\omega \\
&=\int_\R \left(\overline{\int_\R e^{-2\pi I\omega t}\psi(\omega) d\omega}\right)\mathcal{F}_I(\phi)(t)dt\\
&=\int_\R \overline{\mathcal{F}_I(\psi)(t)}\mathcal{F}_I(\phi)(t)dt \\
&= \scal{\mathcal{F}_I(\phi),\mathcal{F}_I(\psi)}_{L^2(\R,\Hq)}. \\
\end{split}
\]

As a direct consequence, we have for any $\phi\in L^2(\R,\Hq)$
\[ \begin{split}
 \displaystyle ||\mathcal{F}_I(\phi)||^2_{L^2(\R,\Hq)} & = \scal{\mathcal{F}_I(\phi),\mathcal{F}_I(\phi)}_{L^2(\R,\Hq)} \\
&=\scal{\phi,\phi}_{L^2(\R,\Hq)}\\
&= ||\phi||^2_{L^2(\R,\Hq)}. \\
\end{split}
\]
\end{proof}
The following remark may be of interest in some other contexts.
  \begin{rem}
  The formal convolution of two given signals $\phi, \psi:\R\longrightarrow \Hq$ when it exists is defined by
  $$\displaystyle (\phi *\psi)(t):=\int_{\R} \phi(\tau)\psi(t-\tau)d\tau.$$
  In particular, if the window function $\phi$ is real valued the 1D QFT satisfies the classical property $$\mathcal{F}_I(\phi*\psi)=\mathcal{F}_I(\phi)\mathcal{F}_I(\psi).$$
 \end{rem}
\section{Quaternion short-time Fourier transform with a Gaussian window}
The idea of the short-time Fourier transform is to obtain information about local properties of the signal $f$. In order to achieve this aim the signal $f$ is restricted to an interval and after its Fourier transform is evaluated. However, since a sharp cut-off can introduce artificial discontinuities and can create problems, it is usually chosen a smooth cut-off function $\varphi$ called "window function".

The aim of this section is to propose a quaternionic analogue of the short-time Fourier transform in dimension one with a Gaussian window function $ \varphi(t)= 2^{1/4} e^{- \pi t^2}$.  For this, we consider the following formula  \cite[Prop. 3.4.1]{G}
\begin{equation}
\label{one}
V_{ \varphi} f(x, \omega)=e^{-\pi i x \omega} Gf(\bar{z}) e^{ \frac{- \pi |z|^2}{2}},
\end{equation}
where the variables $(x, \omega) \in \mathbb{R}^2$ have been converted into a complex vector $z=x+i \omega$, and $Gf(z)$ is the complex version of the Segal-Bargmann transform according to \cite{G}. Therefore, we want to extend \eqref{one} to the quaternionic setting. To this end, we use the quaternionic analogue of the Segal-Bargmann transform \cite{DG} and the slicing representation of the quaternions $q=x+I \omega$, where $I \in \mathbb{S}$.

If the signal is complex we denote the short-time Fourier transform as $V_{\varphi}$, while if the signal is $\Hq$-valued we identify the short-time Fourier transform as $\mathcal{V}_{\varphi}$.
\begin{defn}
Let $f: \mathbb{R} \to \mathbb{H}$ be a function in $ L^{2}(\mathbb{R}, \mathbb{H})$. We define the 1D quaternion short time Fourier transform of $f$ with respect to $  \varphi(t)= 2^{1/4} e^{- \pi t^2}$ as
\begin{equation}
\label{two}
\mathcal{V}_{\varphi} f(x, \omega)=e^{-I \pi x \omega} \mathcal{B}_\Hq^S(f) \biggl(\frac{\bar{q}}{\sqrt{2}} \biggl) e^{-\frac{|q|^2 \pi}{2}},
\end{equation}
where $q= x+I \omega$ and $ \mathcal{B}_\Hq^S(f)(q)$ is the quaternionic Segal-Bargmann transform defined in \eqref{Barg}.
\end{defn}
Using \eqref{Barg} with $ \nu=2 \pi$, we can write \eqref{two} in the following way
\begin{equation}
\label{three}
\mathcal{V}_{\varphi}f(x, \omega)= 2^{\frac{3}{4}} \int_{\mathbb{R}} e^{- \pi \bigl( \frac{\bar{q}^2}{2}+t^2 \bigl)+2 \pi \bar{q} t- I \pi x \omega - \frac{|q|^2 \pi}{2}} f(t) \, dt.
\end{equation}
From this formula we are able to put in relation the 1D quaternion short-time Fourier transform and the 1D quaternion Fourier transform defined in section 3.

\begin{lem}
\label{SF2}
Let $f$ be a function in $L^2(\mathbb{R}, \mathbb{H})$ and $ \varphi(t)=2^{1/4} e^{- \pi t^2}$, recalling the 1D quaternion Fourier transform we have
\begin{equation}
\label{four}
\mathcal{V}_{\varphi} f(x, \omega)=\sqrt{2} \mathcal{F}_{I} (f \cdot \tau_x \varphi)( \omega).
\end{equation}
\end{lem}
\begin{proof}
By putting $q=x+I \omega$ in \eqref{three} we have
\begin{eqnarray*}
\mathcal{V}_{\varphi}f(x, \omega)&=& 2^{\frac{3}{4}} e^{-I \pi x \omega} e^{- \frac{x^2 \pi}{2}} e^{- \frac{\omega^2 \pi}{2}} \int_{\mathbb{R}} e^{- \pi t^2} e^{- \frac{\pi}{2} \bigl( x^2- \omega^2-2x \omega I \bigl)} \cdot \\ \nonumber
&& \cdot e^{2 \pi (x-I \omega) t} f(t)\, dt\\
&=& 2^{\frac{3}{4}} \int_{\mathbb{R}} e^{- \pi t^2 - \pi x^2+2 \pi x t} e^{-2 \pi I \omega t} f(t) \, dt \\
&=& \sqrt{2} \int_{\mathbb{R}} e^{-2 \pi I \omega t} f(t) 2^{\frac{1}{4}}e^{- \pi (t-x)^2} \, dt\\
&=& \sqrt{2} \int_{\mathbb{R}} e^{-2 \pi I \omega t} f(t) \varphi(t-x) \, dt=\sqrt{2} \mathcal{F}_{I} (f \cdot \tau_x \varphi)( \omega).
\end{eqnarray*}
\end{proof}
Now, we prove a formula which relates the 1D quaternion Fourier transform and its signal through the 1D short-time Fourier transform.
\begin{prop}
If $ \varphi$ is a Gaussian function $ \varphi(t)=2^{1/4} e^{- \pi t^2}$ and $f \in L^2(\mathbb{R}, \mathbb{H})$  then
\begin{equation}
\label{PS1}
\mathcal{V}_{\varphi}f(x, \omega)= \sqrt{2} e^{- 2 \pi I \omega x} \mathcal{V}_{\varphi} \mathcal{F}_I(f)(\omega, -x).
\end{equation}
\end{prop}
\begin{proof}
Recalling the definition of modulation and of inner product on $L^2(\mathbb{R}, \mathbb{H})$, by Lemma \ref{SF2} we have
\begin{eqnarray}
\label{PS2}
\mathcal{V}_{\varphi}f(x, \omega)&=& \sqrt{2}	\int_{\mathbb{R}} \overline{ e^{2 \pi I \omega t} \varphi(t-x)} f(t) \, dt\\
&=& \sqrt{2} \int_{\mathbb{R}} \overline{ M_{\omega} \tau_x \varphi(t)} f(t) \, dt= \sqrt{2} \scal{f, M_\omega \tau_x \varphi}. \nonumber
\end{eqnarray}
Using the Plancherel theorem for the 1D quaternion Fourier transform, the property \eqref{F3} and the fact that $ \mathcal{F}_I(\varphi)= \varphi$ we have
\begin{eqnarray*}
\mathcal{V}_{\varphi}f(x, \omega)&=& \sqrt{2} \scal{\mathcal{F}_I(f), \mathcal{F}_I (M_\omega \tau_x \varphi)}\\
&=& \sqrt{2} \scal{\mathcal{F}_I(f), \tau_\omega M_{-x}   \mathcal{F}_I (\varphi)}\\
&=& \sqrt{2} \scal{\mathcal{F}_I(f), \tau_\omega M_{-x} \varphi}
\end{eqnarray*}
Finally, from \eqref{Com} and \eqref{PS2} we get
$$ \mathcal{V}_{\varphi}f(x, \omega)=\sqrt{2} e^{-2 \pi I \omega x}  \scal{\mathcal{F}_I(f),  M_{-x} \tau_\omega \varphi}= \sqrt{2} e^{-2 \pi I \omega x} \mathcal{V}_{\varphi}\mathcal{F}_I(f)(\omega, -x).$$
\end{proof}

\subsection{Moyal fromula}
Now, we prove the Moyal formula and an isometric relation for the 1D quaternion short-time Fourier transform in two ways. In the first way we use the properties of the quaternionic Segal- Bargmann transform, whereas in the second way we use Lemma \ref{SF2} and some basic properties of 1D quaternion Fourier transform.

\begin{prop}
For any $f \in L^2(\mathbb{R}, \mathbb{H})$
\begin{equation}
\label{six}
\| \mathcal{V}_\varphi f \|_{L^2(\R^2,\Hq)}= \sqrt{2} \| f \|_{L^2(\R,\Hq)}.
\end{equation}
\end{prop}
\begin{proof}
We use the slicing representation of the quaternions $ q=x+I \omega$ and  formula \eqref{two} to get
\begin{eqnarray*}
\| \mathcal{V}_\varphi f \|_{L^2(\mathbb{R}, \mathbb{H})}^2 &=& \int_{\mathbb{R}^2} | \mathcal{V}_\varphi f(x, \omega)|^2 \, d \omega \, dx \\
&=&  \int_{\mathbb{R}}|e^{- I \pi x \omega}|^2 \biggl| \mathcal{B}_\Hq^S(f) \biggl(\frac{\bar{q}}{\sqrt{2}}  \biggl) \biggl|^2 e^{-|q|^2 \pi}  \, d \omega \, dx\\
&=& \int_{\mathbb{R}}\biggl| \mathcal{B}_\Hq^S(f) \biggl(\frac{\bar{q}}{\sqrt{2}}  \biggl) \biggl|^2 e^{-|q|^2 \pi}  \, d \omega \, dx.
\end{eqnarray*}
Now, using the change of variable $p= \frac{\bar{q}}{\sqrt{2}}$ we have that $dA(p)= \frac{1}{2}  \, d \omega \, dx$, hence by \cite[Thm. 4.6]{DG}
\begin{eqnarray*}
\| \mathcal{V}_\varphi f \|_{L^2(\mathbb{R}, \mathbb{H})}^2 &=& 2 \int_{\mathbb{R}^2} |\mathcal{B}_\Hq^S(f)(p)|^2 e^{-2 \pi |q|^2}\, dA(p)\\
&=& 2 \| \mathcal{B}_\Hq^S(f) \|_{\mathcal{F}^{2, 2 \pi}_{Slice}}^2=2 \|f \|_{L^2(\mathbb{R}, \mathbb{H})}^2.
\end{eqnarray*}
Therefore
$$ \|\mathcal{V}_\varphi f \|_{L^2(\mathbb{R}, \mathbb{H})}= \sqrt{2}  \|f \|_{L^2(\mathbb{R}, \mathbb{H})}.$$
\end{proof}
Thus, the 1D quaternionic short-time Fourier transform is an isometry from $L^2(\mathbb{R}, \mathbb{H})$ into $ L^2( \mathbb{R}^2, \mathbb{H})$.

\begin{prop}[Moyal formula]
Let $f,g$  be functions in $L^2(\mathbb{R}, \mathbb{H})$. Then we have
\begin{equation}
\label{five}
\langle \mathcal{V}_\varphi f, \mathcal{V}_\varphi g \rangle_{L^2( \mathbb{R}^2, \mathbb{H})} =2 \langle f,g \rangle_{L^2(\mathbb{R}, \mathbb{H})}.
\end{equation}
\end{prop}
\begin{proof}
From \eqref{two} we get
\begin{eqnarray*}
\langle \mathcal{V}_\varphi f, \mathcal{V}_\varphi g \rangle_{L^2( \mathbb{R}^2, \mathbb{H})} &=& \int_{\mathbb{R}^2} \overline{\mathcal{V}_{\varphi}g(x, \omega)} \mathcal{V}_\varphi f(x, \omega)  \, d \omega \, dx\\
&=& \int_{\mathbb{R}^2} \overline{e^{-I \pi x \omega} \mathcal{B}_\Hq^S(g) \biggl(\frac{\bar{q}}{\sqrt{2}} \biggl) e^{-\frac{|q|^2 \pi}{2}}} e^{-I \pi x \omega} \cdot \\ \nonumber
&& \cdot \mathcal{B}_\Hq^S(f) \biggl(\frac{\bar{q}}{\sqrt{2}} \biggl) e^{-\frac{|q|^2 \pi}{2}}  \, d \omega \, dx\\
&=& \int_{\mathbb{R}^2} \overline{ \mathcal{B}_\Hq^S(g) \biggl(\frac{\bar{q}}{\sqrt{2}} \biggl) } e^{I \pi x \omega} e^{-I \pi x \omega} \cdot \\ \nonumber
&& \cdot \mathcal{B}_\Hq^S(f) \biggl(\frac{\bar{q}}{\sqrt{2}} \biggl) e^{-|q|^2 \pi} \, d \omega \, dx\\
&=& \int_{\mathbb{R}^2} \overline{ \mathcal{B}_\Hq^S(g) \biggl(\frac{\bar{q}}{\sqrt{2}} \biggl) } \mathcal{B}_\Hq^S(f) \biggl(\frac{\bar{q}}{\sqrt{2}} \biggl) e^{-|q|^2 \pi}  \, d \omega \, dx.
\end{eqnarray*}
Using the same change of variables as before $p= \frac{\bar{q}}{\sqrt{2}}$ and from \eqref{NL} we obtain
\begin{eqnarray*}
\langle \mathcal{V}_\varphi f, \mathcal{V}_\varphi g \rangle_{L^2( \mathbb{R}^2, \mathbb{H})} &=& 2 \int_{\mathbb{R}^2} \overline{ \mathcal{B}_\Hq^S(g) (p)} \mathcal{B}_\Hq^S(f) (p) e^{-2|q|^2 \pi} \, d \omega \, dx \\
&=& 2 \langle \mathcal{B}_\Hq^S(f), \mathcal{B}_\Hq^S(g) \rangle_{\mathcal{F}^{2,2\pi}_{Slice}(\mathbb{H})}= 2 \langle f,g \rangle_{L^{2}(\mathbb{R}, \mathbb{H})}.
\end{eqnarray*}
\end{proof}

\begin{rem}
If we put $f= \frac{h_k^{2 \pi}(t)}{\| h_k^{2 \pi}(t) \|_2^2}$ in \eqref{two} by \cite[Lemma 4.4]{DG} we get
$$\mathcal{V}_\varphi f(x, \omega)= e^{-I \pi x \omega} e^{- \frac{\pi}{2} |q|^2} \frac{2^{3/4}}{2^k k!} \bar{q}^k.$$
\end{rem}

\begin{rem}
From \eqref{four} we can prove \eqref{five} in another way. This proof may be of interest in some other contexts.
	
Let us assume $f,g \in L^2(\mathbb{R}, \mathbb{H})$ and recall $ \varphi(t)= 2^{1/4} e^{- \pi t^2}$, by Lemma \ref{SF2} and Plancherel theorem for the 1D quaternion Fourier transform  we have
\begin{eqnarray*}
\langle \mathcal{V}_\varphi f, \mathcal{V}_\varphi g \rangle_{L^{2}(\mathbb{R}^2, \mathbb{H})}&=& \int_{\mathbb{R}^2} \overline{\mathcal{V}_\varphi g(x, \omega)} \mathcal{V}_\varphi f(x, \omega) \, d \omega \, dx\\
&=& 2 \int_{\mathbb{R}^2} \overline{\mathcal{F}_{I} (g \cdot \tau_x \varphi)( \omega)}  \mathcal{F}_{I} (f \cdot \tau_x \varphi)( \omega) \, d \omega \, dx\\
&=& 2 \int_{\mathbb{R}^2} \overline{g(\omega) \cdot \tau_x \varphi( \omega)}   f(\omega) \cdot \tau_x \varphi ( \omega) \, d \omega \, dx.
\end{eqnarray*}
Now, by Fubini's theorem  and the fact that $ \| \varphi \|_2^2=1$ we get
\begin{eqnarray*}
\langle \mathcal{V}_\varphi f, \mathcal{V}_\varphi g \rangle_{L^{2}(\mathbb{R}^2, \mathbb{H})}&=&2 \int_{\mathbb{R}} \biggl( \int_{\mathbb{R}} \overline{g(\omega) \cdot \tau_x \varphi( \omega)}   f(\omega) \cdot \tau_x \varphi ( \omega) \, dx  \biggl) \, d \omega\\
&=& 2 \int_{\mathbb{R}} \biggl( \int_{\mathbb{R}} \overline{g( \omega)} f( \omega) \varphi^2(x- \omega) \, dx \biggl) \, d \omega\\
&=& 2 \int_{\mathbb{R}} \overline{g( \omega)} f( \omega)  \biggl( \int_{\mathbb{R}}\varphi^2(x- \omega) \, dx \biggl) \, d \omega\\
&=& 2 \int_{\mathbb{R}} \overline{g( \omega)} f( \omega)  \| \varphi \|_2^2 \, d \omega
=  2 \int_{\mathbb{R}} \overline{g( \omega)} f( \omega)   \, d \omega\\
&=& 2 \langle f, g \rangle_{L^{2}(\mathbb{R}, \mathbb{H})}.
\end{eqnarray*}
Hence
\begin{equation}
\label{seven}
\langle \mathcal{V}_\varphi f, \mathcal{V}_\varphi g \rangle_{L^2( \mathbb{R}^2; \mathbb{H})} =2 \langle f,g \rangle_{L^2(\mathbb{R}, \mathbb{H})} .
\end{equation}
If we put $f=g$ in \eqref{seven} we obtain \eqref{six}.
\end{rem}

\subsection{Inversion formula and adjoint of QSTFT}
The 1D QSTFT with Gaussian window $\varphi$ satisfies a reconstruction formula that we  prove in the following.
\begin{thm}
Let $f\in L^2(\R,\Hq)$. Then, we have
$$f(y)=\displaystyle 2^{-\frac{1}{4}} \int_{\R^2} e^{2\pi I \omega y}\mathcal{V}_{\varphi}f(x,\omega)e^{-\pi(y-x)^2}dx d\omega, \textbf{  } \forall y\in\R.$$
\end{thm}
\begin{proof}
For all $y\in\R$, we set $$\displaystyle g(y)=2^{-\frac{1}{4}}\int_{\R^2} e^{2\pi I \omega y}\mathcal{V}_{\varphi}f(x,\omega)e^{-\pi(y-x)^2}dx d\omega.$$
Let $h\in L^2(\R,\Hq)$. Fubini's theorem combined with Moyal formula for QSTFT leads to
\[ \begin{split}
 \displaystyle  \scal{g,h}_{L^2(\R, \mathbb{H})} & = \int_\R \overline{h(y)}g(y)dy \\
&=2^{-\frac{1}{4}}\int_{\R^3} \overline{h(y)}e^{2\pi I \omega y}\mathcal{V}_{\varphi}f(x,\omega)e^{-\pi(y-x)^2}dx d\omega dy\\
&=2^{-1} \sqrt{2}\int_{\R^2} \left(\overline{\int_\R e^{-2\pi I \omega y} 2^{\frac{1}{4}}e^{-\pi(y-x)^2}h(y)dy}\right)\mathcal{V}_{\varphi}f(x,\omega)dx d\omega \\
&= 2^{-1} \int_{\R^2} \overline{\mathcal{V}_{\varphi}h(x,\omega)}\mathcal{V}_{\varphi}f(x,\omega)dx d\omega \\
&=  2^{-1} \scal{\mathcal{V}_{\varphi}f,\mathcal{V}_{\varphi}h}_{L^2(\R^2)} \\
&= \scal{f,h}_{L^2(\R, \mathbb{H})}. \\
\end{split}
\]
Hence, we have $$\displaystyle f(y)=g(y)=2^{-\frac{1}{4}}\int_{\R^2} e^{2\pi I \omega y}\mathcal{V}_{\varphi}f(x,\omega)e^{-\pi(y-x)^2}dx d\omega.$$

This ends the proof.
\end{proof}
We note that the QSTFT admits a left side inverse that we can compute as follows
\begin{thm}
Let $\varphi$ denote the Gaussian window $ \varphi(t)= 2^{1/4} e^{- \pi t^2}$ and let us consider the operator $\mathcal{A}_\varphi:L^2(\R^2,\Hq)\longrightarrow L^2(\R,\Hq)$ defined for any $F\in L^2(\R^2,\Hq)$  by
$$\displaystyle \mathcal{A}_{\varphi}(F)(y)= 2^{\frac{3}{4}}\int_{\R^2} e^{2\pi I \omega y}F(x,\omega)e^{-\pi(y-x)^2}dx d\omega, \textbf{  } \forall y\in\R.$$

Then, $\mathcal{A}_\varphi$ is the adjoint of $\mathcal{V}_{\varphi}$. Moreover, the following identity holds \begin{equation}
\label{twelve}
 \mathcal{V}_\varphi^*\mathcal{V}_\varphi=2Id.
\end{equation}
\end{thm}
\begin{proof}
Let $F\in L^2(\R^2,\Hq)$ and $h\in L^2(\R,\Hq)$. We use some calculations similar to the previous result and get
\[ \begin{split}
 \displaystyle  \scal{\mathcal{A}_\varphi(F),h}_{L^2(\R, \Hq)} & = \int_\R \overline{h(y)}\mathcal{A}_\varphi(F)(y)dy \\
&=2^{\frac{3}{4}}\int_{\R^3} \overline{h(y)}e^{2\pi I \omega y}F(x,\omega)e^{-\pi(y-x)^2}dx d\omega dy\\
&=\int_{\R^2} \sqrt{2} \left(\overline{\int_\R e^{-2\pi I \omega y} 2^{\frac{1}{4}}e^{-\pi(y-x)^2}h(y)dy}\right)F(x,\omega)dx d\omega \\
&= \int_{\R^2} \overline{\mathcal{V}_{\varphi}h(x,\omega)}F(x,\omega)dx d\omega \\
&= \scal{F,\mathcal{V}_{\varphi}h}_{L^2(\R^2, \Hq)}. \\
\end{split}
\]

In particular, this shows that $$\displaystyle \mathcal{A}(\varphi)(F)=\mathcal{V}_\varphi^*(F), \textbf{  } \forall F\in L^2(\R^2, \Hq).$$
From reconstruction formula we obtain \eqref{twelve}.

\end{proof}
\begin{rem}
We note that the identity $\mathcal{V}_\varphi^*\mathcal{V}_\varphi= 2Id$ provides another proof for the fact that QSTFT is an isometric operator and the adjoint $\mathcal{V}_\varphi^*$ defines a left inverse.
\end{rem}

\subsection{The eigenfunctions of the 1D quaternion Fourier transform}
Through the 1D QSTFT we can prove in another way that the eigenfunctions of the 1D quaternion Fourier transform are given by the Hermite functions.
\begin{prop}
The Hermite functions $ h_k^{2 \pi}(t)$ are eigenfunctions of the 1D quaternion Fourier transform :
$$ \mathcal{F}_I (h_k^{2 \pi})(t)=2^{-1/2} (-I)^k h_k^{2 \pi}(t), \qquad t \in \mathbb{R}.$$
\end{prop}
\begin{proof}
By identity \eqref{two} and \cite[Lemma 4.4]{DG} we have
\begin{eqnarray}
\label{eleven}
\mathcal{V}_{\varphi}(h_k^{2 \pi})(x, - \omega)&=&e^{I \pi x \omega} \mathcal{B}_{\Hq}^{S}(h_k^{2 \pi}) \bigl( \frac{q}{\sqrt{2}} \bigl) e^{- \frac{ \pi |q|^2}{2}}\\ \nonumber
&=& e^{I \pi x \omega} 2^{1/4} 2^{k/2} (2 \pi)^k 2^{-k/2} q^k e^{- \frac{ \pi |q|^2}{2}}\\
&=& e^{I \pi x \omega} 2^{1/4} (2 \pi)^k q^k e^{- \frac{ \pi |q|^2}{2}}. \nonumber
\end{eqnarray}
Recalling that $q= x+I \omega$ and using \eqref{PS1} we obtain
\begin{eqnarray*}
\mathcal{V}_{\varphi} \mathcal{F}_I (h_k^{2 \pi})(x,- \omega)&=&2^{-1/2} e^{2 \pi I \omega x} \mathcal{V}_{\varphi}h_k^{2 \pi} (\omega, x)\\
&=& 2^{-1/2} e^{2 \pi I \omega x}  e^{-I \pi \omega x} \mathcal{B}_{\Hq}^S (h_k^{2 \pi}) \biggl( \frac{\omega -Ix}{ \sqrt{2}} \biggl) e^{- \frac{|q|^2 \pi}{2}}\\
&=& 2^{-1/2} e^{\pi I \omega x}   \mathcal{B}_{\Hq}^S (h_k^{2 \pi}) \biggl( \frac{-Iq}{ \sqrt{2}} \biggl) e^{- \frac{|q|^2 \pi}{2}}\\
&=& 2^{-1/2} e^{\pi I \omega x} 2^{1/4} 2^{k/2} (2 \pi)^k (-I)^k 2^{-k/2} q^k e^{- \frac{|q|^2 \pi}{2}}\\
&=& 2^{-1/2} (-I)^k e^{ I \pi \omega x} 2^{1/4} (2 \pi)^k  q^k e^{- \frac{|q|^2 \pi}{2}}.
\end{eqnarray*}
Combining with \eqref{eleven}
$$ \mathcal{V}_{\varphi} \mathcal{F}_I (h_k^{2 \pi})(x,- \omega)= 2^{-1/2} (-I)^k \mathcal{V}_{\varphi}h_k^{2 \pi}(x, - \omega).$$
From \eqref{twelve} we know that $V_{\varphi}$ is injective, hence we have the thesis.
\end{proof}

\subsection{Reproducing kernel property}
The inversion formula gives us the possibility to write the 1D QSTFT using the reproducing kernel associated to the quaternion Gabor space, introduced in \cite{AF}, with a Gaussian window that is defined by

$$\mathcal{G}^{\varphi}_{\Hq}:=\lbrace{\mathcal{V}_\varphi f, \textbf{  } f\in L^2(\R,\Hq)}\rbrace.$$
\begin{thm}
Let $f$ be in $L^2( \mathbb{R, \mathbb{H}})$ and $ \varphi(t)=2^{1/4} e^{- \pi t^2}$. If
$$ \mathbb{K}_{\varphi}(\omega, x; \omega', x')=\int_{\mathbb{R}} e^{-2 \pi I \omega' t} \varphi(t-x') \overline{e^{-2 \pi I \omega t} \varphi(t-x)}  \, dt,$$
then $ \mathbb{K}_{\varphi}(\omega, x; \omega', x')$ is the reproducing kernel i.e.
$$ \mathcal{V}_{\varphi}f(x', \omega')= \int_{\mathbb{R}^2} \mathbb{K}_{\varphi}(\omega, x; \omega', x') \mathcal{V}_{\varphi}f(x, \omega) \, dx d \omega.$$
\end{thm}
\begin{proof}
By Lemma \ref{SF2} and the reconstruction formula we have
\begin{eqnarray*}
\mathcal{V}_{\varphi}f(x', \omega')&=&2^{3/4} \int_{\mathbb{R}} e^{-2 \pi I \omega' t} f(t) e^{- \pi(t-x')^2} \, dt\\
&=& 2^{3/4} \int_{\mathbb{R}} e^{-2 \pi I \omega' t} e^{- \pi(t-x')^2} 2^{-\frac{1}{4}}\cdot \\ \nonumber
&& \cdot \biggl( \int_{\mathbb{R}^2} e^{2 \pi I \omega t}e^{- \pi (t-x)^2} \mathcal{V}_{\varphi} f (x, \omega) \, d x \, d \omega \biggl) \, dt\\
&=& \sqrt{2} \int_{\mathbb{R}^3} e^{-2 \pi I (\omega'- \omega) t} e^{- \pi(t-x')^2} e^{- \pi (t-x)^2} \cdot \\ \nonumber
&& \cdot  \mathcal{V}_{\varphi} f (x, \omega)  \, dx \, d \omega  \, dt.
\end{eqnarray*}
Using Fubini's theorem we have
\begin{eqnarray*}
\mathcal{V}_{\varphi}f(x', \omega')&=& \sqrt{2} \int_{\mathbb{R}^2}   \biggl(\int_{\mathbb{R}} e^{-2 \pi I (\omega'- \omega) t} e^{- \pi(t-x')^2} e^{- \pi (t-x)^2} \, dt  \biggl) \cdot \\ \nonumber
&& \cdot \mathcal{V}_{\varphi} f (x, \omega)\, dx \, d \omega  \\
&=& \int_{\mathbb{R}^2}  \biggl(\int_{\mathbb{R}} e^{-2 \pi I \omega' t} 2^{1/4} e^{- \pi(t-x')^2} \overline{2^{1/4} e^{-2 \pi I \omega t} e^{- \pi (t-x)^2}}  \, dt  \biggl) \cdot \\
&& \cdot \mathcal{V}_{\varphi} f (x, \omega)\, dx \, d \omega  \\
&=& \int_{\mathbb{R}^2}  \biggl(\int_{\mathbb{R}} e^{-2 \pi I \omega' t} \varphi(t-x') \overline{e^{-2 \pi I \omega t} \varphi(t-x)}  \, dt  \biggl) \cdot \\ \nonumber
&& \cdot \mathcal{V}_{\varphi} f (x, \omega) \, dx \, d \omega\\
&=& \int_{\mathbb{R}^2} \mathbb{K}_{\varphi}(\omega, x; \omega', x') \mathcal{V}_{\varphi}f(x, \omega) \, dx d \omega.
\end{eqnarray*}
\end{proof}

\subsection{Lieb's uncertainty principle for QSTFT}
The QSTFT follows the Lieb's uncertainty principle with some weak differences comparing to the classical complex case. Indeed, we first study the weak uncertainty principle which is the subject of this result
\begin{thm}[Weak uncertainty principle]
Let $f\in L^2(\R,\Hq)$ be a unit vector (i.e $||f||=1$), $U$ an open set of $\R^2$ and $\varepsilon \geq 0$ such that $$\displaystyle\int_U |\mathcal{V}_\varphi f(x,\omega)|^2dxd\omega \geq 1-\varepsilon.$$
Then, we have $$|U|\geq \frac{1-\varepsilon}{2},$$
where $|U|$ denotes the Lebesgue measure of $U$.

\end{thm}
\begin{proof}
We note that using Definition of QSTFT and \cite[Prop. 4.3]{DG} we obtain \[ \begin{split}
 \displaystyle  |\mathcal{V}_{\varphi}f(x,\omega)| & =  |\mathcal{B}_\Hq^Sf(\bar{q}/\sqrt{2})|e^{-\frac{|q|^2}{2}\pi} \\
& = |\mathcal{B}_\Hq^Sf(p)|e^{-\pi|p|^2}; \text{  } p=\bar{q}/\sqrt{2}\\
&\leq \sqrt{2}||f||_{L^2(\R)}. \\
\end{split}
\]
Thus, by hypothesis we get $$\displaystyle 1-\varepsilon \leq \int_U |\mathcal{V}_\varphi f(x,\omega)|^2dxd\omega  \leq ||\mathcal{V}_\varphi f||_\infty ^2 |U|\leq 2|U|.$$
Hence, we have $$|U|\geq \frac{1-\varepsilon}{2}.$$
\end{proof}
\begin{thm}[Lieb's inequality]
Let $f\in L^2(\R,\Hq)$ and $2 \leq p <\infty$. Then, we have $$\displaystyle \int_{\R^2}|\mathcal{V}_{\varphi}f(x,\omega)|^pdx d\omega\leq \frac{2^{p+1}}{p}||f||_{L^2(\R,\Hq)}^p$$

\end{thm}
\begin{proof}

Let $I,J\in\Sq$ be such that $I$ is orthogonal to $J$. Then, for $f\in L^2(\R,\Hq)$, there exist $f_1,f_2\in L^2(\R,\C_I)$ such that $$f(t)=f_1(t)+f_2(t)J, \textbf{  } \forall t\in \R$$
and for which the classical Lieb's inequality \cite{L} holds , i.e:
$$\displaystyle \int_{\R^2}|V_{\varphi}f_l(x,\omega)|^pdx d\omega\leq \frac{2}{p}||f_l||_{L^2(\R,\C_I)}^p; \textbf{  } l=1,2.$$
In particular, by definition of QSTFT we have
$$\mathcal{V}_\varphi f(x,\omega)=\mathcal{V}_\varphi f_1(x,\omega)+V_\varphi f_2(x,\omega)J, \textbf{ } \forall (x,\omega)\in \R^2.$$
Thus,
\[ \begin{split}
 \displaystyle |\mathcal{V}_\varphi f(x,\omega)|^p & \leq  \left(|V_\varphi f_1(x,\omega)|+|V_\varphi f_2(x,\omega)|\right )^p\\
& \leq 2^{p-1}\left(|V_\varphi f_1(x,\omega)|^p+|V_\varphi f_2(x,\omega)|^p\right ).
\end{split}
\]
We use the classical Lieb's inequality on each component combined with the fact that $||f_l||_p\leq ||f||_p$ for $l=1,2$ and get
\[ \begin{split}
 \displaystyle \int_{\R^2}|\mathcal{V}_\varphi f(x,\omega)|^pdxd\omega  & \leq \frac{2^p}{p} \left(||f_1||_{L^2(\R)}^p+||f_2||_{L^2(\R)}^p\right )\\
& \leq \frac{2^{p+1}}{p}||f||_{L^2(\R,\Hq)}^p.
\end{split}
\]
This ends the proof.
\end{proof}
The next result improves the weak uncertainty principle in the sense that it gives a best sharper estimate for $|U|$.
\begin{thm}
Let $f\in L^2(\R,\Hq)$ be a unit vector, $U$ an open set of $\R^2$ and $\varepsilon \geq 0$ such that $$\displaystyle\int_U |\mathcal{V}_\varphi f(x,\omega)|^2dxd\omega \geq 1-\varepsilon.$$
Then, we have $$|U|\geq c_p (1-\varepsilon)^{\frac{p}{p-2}},$$
where $|U|$ denotes the Lebesgue measure of $U$ and $c_p=\left(\frac{2^{p+1}}{p}\right)^{-\frac{2}{p-2}}$.
\begin{proof}
Let $f\in L^2(\R,\Hq)$ be such that $||f||_{L^2(\R,\Hq)}=1$. We first apply  Holder inequality with exponents $\displaystyle q=\frac{p}{2}$ and $\displaystyle q'=\frac{p}{p-2}$. Then, using Lieb's inequality for QSTFT we get
\[ \begin{split}
 \displaystyle \int_{U}|\mathcal{V}_\varphi f(x,\omega)|^2 dxd\omega  & = \int_{\R^2}|\mathcal{V}_\varphi f(x,\omega)|^2\chi_{_{U}}(x,\omega) dxd\omega\\
& \leq \left(\int_{\R^2}|\mathcal{V}_\varphi f(x,\omega)|^p dxd\omega \right)^{\frac{2}{p}}|U|^{\frac{p-2}{p}}  \\
& \leq \left(\frac{2^{p+1}}{p}\right)^{\frac{2}{p}}|U|^{\frac{p-2}{p}}. \\
\end{split}
\]
Hence, by hypothesis we obtain $$|U|\geq c_p (1-\varepsilon)^{\frac{p}{p-2}}$$ where $c_p=\left(\frac{2^{p+1}}{p}\right)^{-\frac{2}{p-2}}$.

\end{proof}

\end{thm}
\section{Concluding remarks}
In this paper, we studied a quaternion short-time Fourier transform (QSTFT) with a Gaussian window. This window function corresponds to the first normalized Hermite function given by $\psi_0(t)=\varphi(t)=2^{1/4} e^{- \pi t^2}$. Based on the quternionic Segal-Bargmann transform we proved several results  including different versions of Moyal formula, reconstruction formula, Lieb's principle, etc.  A more general problem in this framework is to consider a QSTFT associated to some generic quaternion valued window $\psi$. For a given quaternion $q=x+I \omega$ we plan to investigate in our future research  works the properties of the QSTFT defined for any $f\in L^2(\R,\Hq)$ by
$$ \mathcal{V}_\psi f(x, \omega)= \int_{\mathbb{R}}e^{-2 \pi I t \omega}  \overline{\psi(t-x)}  f(t) dt.$$
In particular, studying such transforms with normalized Hermite functions \\ $\lbrace \psi_n(t) \rbrace_{n\geq 0}$ that are real valued windows will be related to the theory of slice poly-analytic functions on quaternions considered in \cite{ADS}.
\newline
\newline
\newline
\textbf{Acknowledgements}
\newline
We would like to thank Prof. Irene Sabadini for reading an earlier version of this paper and for her interesting comments. The second author acknowledges the support of the project INdAM Doctoral Programme in Mathematics and/or Applications Cofunded by Marie Sklodowska-Curie Actions, acronym: INdAM-DP-COFUND-2015, grant number: 713485.

\hspace{2mm}

\noindent
Antonino De Martino,
Dipartimento di Matematica \\ Politecnico di Milano\\
Via Bonardi n.~9\\
20133 Milano\\
Italy

\noindent
\emph{email address}: antonino.demartino@polimi.it\\
\emph{ORCID iD}: 0000-0002-8939-4389

\vspace*{5mm}
\noindent
Kamal Diki,
Dipartimento di Matematica \\ Politecnico di Milano\\
Via Bonardi n.~9\\
20133 Milano\\
Italy

\noindent
\emph{email address}: kamal.diki@polimi.it\\
\emph{ORCID iD}: 0000-0002-4359-7535


\begin{thebibliography}{99}

\bibitem[1]{AF} Abreu L.D,  Feichtinger H.G , \emph{Function spaces of Polyanalytic Functions}, Harmonic and complex analysis and its applications, 1–38,Trends Math., Birkhäuser/Springer, Cham, (2014).
\bibitem[2]{ACS} Alpay D., Colombo F., Sabadini I., Salomon G., {\it The Fock space in the slice hyperholomorphic Setting}. In Hypercomplex Analysis: New perspectives and applications. Trends Math. 43--59. (2014).
\bibitem[3]{ACS1} Alpay D., Colombo F., Sabadini I., \emph{ Slice Hyperholomorphic Schur Analysis}, Volume 256 of {\em Operator Theory: Advances and Applications}. Birkh\"{a}user, Basel, (2017).
\bibitem[4]{ADS} Alpay D., Diki K., Sabadini I., {\it On slice polyanalytic functions of a quaternionic variable}. Results Math 74, 17 (2019).
\bibitem[5]{BA} Bahri M., Ashino R., \emph{Two-dimensional quaternionic wiondow Fourier Transform}, in Fourier Transforms - Approach to Scientific Principles, InTechOpen (G.S. Nikolic), 2011.
\bibitem[6]{B} Bargmann V., \emph{On a Hilbert space of analytic functions and an associated integral transform}, Comm. Pure Appl. Math. 14, 187-214. (1961).
\bibitem[7]{CK} Cerejeiras P., Kähler U., \emph{ Monogenic Signal Theory}, in Operator Theory, Springer (D.Alpay), Basel,(2014).
\bibitem[8]{CGK} Colombo F., Gantner J.,; Kimsey, P. \emph{Spectral theory on the S-spectrum for
quaternionic operators}, Operator Theory: Advances and Applications, 270.  Birkhäuser/Springer,
Cham, 2018. ix+356 pp.
\bibitem[9]{CG}  Colombo F., Gantner J., \emph{Quaternionic closed operators, fractional powers and fractional diffusion processes}, Operator Theory: Advances and Applications, 274.  Birkhäuser/Springer, Cham, 2019.
viii+322 pp.
\bibitem[10]{CSS} Colombo F., Sabadini I., Sommen F., \emph{On the Bargmann-Radon transfrom in the monogenic setting}, J. Geom. Phys. 120, 306-316 (2017).
\bibitem[11]{CSS3} Colombo F., Sabadini I., Sommen F., \emph{The Fueter mapping theorem in integral form and the F-functional calculus}, Math. Methods Appl. Sci. 33, 2050-2066 (2010).
\bibitem[12]{CSS4} Colombo F., Sabadini I., Sommen F., \emph{The inverse Fueter mapping theorem}, Commun. Pure Appl. Anal. 10, 1165-1181 (2011).
\bibitem[13]{CSS1} Colombo F, Sabadini I., Struppa D.C., \emph{Noncommutative functional calculus, Progress in Mathematics}, vol. 289, Birkhäuser/Springer Basel AG, Basel, (2011).
\bibitem[14]{CSS2} Colombo F., Sabadini I., Struppa D.C, \emph{Entire Slice Regular Functions}, SpringerBriefs in Mathematics, Springer, Cham, (2016).
\bibitem[15]{D} De Bie H., \emph{Fourier Transforms in Clifford analysis}, in Operator Theory, Springer (D.Alpay), Basel,(2014).
\bibitem[16]{D1} De Martino A., \emph{On the Clifford short-time Fourier transform and its properties} (in preparation).
\bibitem[17]{DKS} Diki K., Krausshar R.S., Sabadini I., \emph{On the Bargmann-Fock-Fueter and Bergman-Fueter integral transfrom}, J. Math Phys. 60, 1-26 (2019).
\bibitem[18]{DG} Diki K., Ghanmi A., \emph{A quaternionic analogue for the Segal-Bargamann transfrom}, Complex. Anal. Oper. Theory 11, 457-473 (2017).
\bibitem[19]{ELS} Ell, T.A.. Le Bihan N., Sangwine S.J., \emph{Quaternion Fourier transform for signal andimage processing, Focus Series in Digital Signal and Image Processing}, London, (2014).
\bibitem[20]{GSS} Gentili G., Stoppato C., Struppa D.C., \emph{Regular functions of a quaternionic varaible}, Springer Monographs, Berlin (2013).
\bibitem[21]{G} Gr\"{o}chening K., \emph{Foundations of Time-Frequency Analysis}, Birkh\"{a}user, Boston, (2001).
\bibitem[22]{L} Lieb E.H , \emph{Integral bounds for radar ambiguity functions and Wigner distribution}, J. Math Phys. 31, 594-599 (1990).
\bibitem[23]{Q} Qian T., \emph{Generalization of Fueters result to $\mathbb{R}^{n+1}$ }, Rend. Mat. Acc. Lincei 9,  111–117 (1997).
\bibitem[24]{N} Neretin Y., \emph{Lectures on Gaussian integral operators and classical groups}, EMS Series of Lectures in Mathematics, European Mathematical Society (EMS), Zürich, (2011). xii+559 pp.
\bibitem[25]{PSS} Pena Pena D., Sabadini I., Sommen F. \emph{Segal-Bargmann-Fock modules of monogenic functions}, J. Math Phys. 58, 103507 (2017).
\bibitem[26]{S}Sce M., \emph{Osservazioni sulle serie di potenze nei moduli quadratici}, Atti Accad. Naz. Lincei. Rend. CI. Sci.Fis. Mat. Nat. 23, 220-225 (1957).
\bibitem[27]{XR} Xu Z., Ren G., \emph{Sharper uncertainty principles in quaternionic Hilbert
spaces}, Math Meth Appl Sci. 43: 1608– 1630 (2020).

\end{thebibliography}
\end{document}